\documentclass[12pt]{amsart}

\usepackage{graphicx}
\usepackage[leqno]{amsmath}
\usepackage{amssymb}
\usepackage[all]{xy}
\usepackage{graphicx}
\usepackage[usenames,dvipsnames]{color}
\usepackage[normalem]{ulem}

\newtheorem{thm}{Theorem}[section]

\newtheorem{lemma}[thm]{Lemma}
\newtheorem{cor}[thm]{Corollary}
\newtheorem{prop}[thm]{Proposition}

\theoremstyle{definition}
\newtheorem{definition}[thm]{Definition}
\newtheorem{example}[thm]{Example}

\date{\today}

\begin{document}

\title[Taut foliations]
{Taut foliations}

\author[Kazez]{William H.  Kazez}
\address{Department of Mathematics, University of Georgia, Athens, GA 30602}
\email{will@math.uga.edu}

\author[Roberts]{Rachel Roberts}
\address{Department of Mathematics, Washington University, St.  Louis, MO 63130}
\email{roberts@math.wustl.edu}


\thanks{This work was partially supported by grants from the Simons Foundation (\#244855 to William Kazez, \#317884 to Rachel Roberts)}


\begin{abstract}
We describe notions of tautness that arise in the study of $C^0$ foliations, $C^{1,0}$ or smoother foliations, and in geometry.  We give examples to show that these notions are different, and discuss how these differences impact some classical foliation results.

We construct examples of smoothly taut $C^{\infty,0}$ foliations that can be $C^0$ approximated by both weakly symplectically fillable, universally tight contact structures and by overtwisted contact structures.
\end{abstract}

\maketitle

\section{Introduction}
In \cite{Sullivan1,Sullivan2}, Sullivan introduced multiple notions of {\sl tautness} of a foliation, and proved the equivalence of these notions for $C^2$ foliations.  The conceptually simplest of these was highlighted further by Gabai (cf, Definition~2.11 of \cite{G1}).  When $T\mathcal F$ is only $C^0$, Definition~2.11 of \cite{G1} admits two natural interpretations.  Moreover, the notions of tautness introduced by Sullivan are distinct when the criterion that $\mathcal F$ be $C^2$ is dropped.

In this paper we describe several notions of tautness that arise in the study of $C^0$ foliations, $C^{1,0}$ or smoother foliations, and in geometry.  We examine how these notions are related, and give examples in \S~\ref{Examples} to show that these notions are different.  

The implications of the differences between versions of tautness are clearest in geometry and contact topology.  We show in Proposition~\ref{no vp flow} that some smoothly taut foliations are not transverse to a volume preserving flow, whereas by Theorem~\ref{vp flow}, $C^{1,0}$ everywhere taut foliations are always transverse to such flows.

In Theorem~\ref{limit of both} we construct smoothly taut $C^{\infty,0}$ foliations that are $C^0$ approximated both by overtwisted contact structures and by everywhere taut foliations.  This contrasts sharply with Theorem~\ref{everwheretautimplies} which shows that $C^{1,0}$ everywhere taut foliations are $C^0$ approximated by, and only by, weakly symplectically fillable, universally tight contact structures.

Since topologically taut $C^0$ foliations are isotopic to everywhere taut $C^{\infty,0}$ foliations (Corollary~\ref{top approx}), the differences between the versions of tautness are unimportant when working with foliations up to topological conjugacy.  In particular, for example, an $L$-space does not admit a transversely orientable, topologically taut, $C^0$ foliation (Corollary~\ref{L space}).

This study was inspired by an observation of Vincent Colin, who pointed out to us that if $T\mathcal F$ is only continuous, then $T\mathcal F$ is not uniquely integrable and this impacts the existence of transversals. We thank Vincent for helpful conversations, and in particular for explaining his work, \cite{ColinFirmo}, where these issues arise in the search for a foliation approximated by a pair of contact structures.  We also thank the Banff International Research Station for their hospitality, which made these conversations possible.

\section{Definitions}

We begin by recalling the definition of foliation, paying careful attention to smoothness.

\begin{definition}\label{folndefn1} Let $M$ be a smooth 3-manifold with empty boundary.  Let $k$ be a non-negative integer or infinity.  A {\sl codimension one foliation}, $\mathcal F$, is a decomposition of $M$ into a disjoint union of connected surfaces, called the {\sl leaves} of $\mathcal F$, together with a collection of charts $U_i$ covering $M$, with $\phi_i:\mathbb R^2 \times \mathbb R \to U_i$ a homeomorphism, such that the preimage of each component of a leaf intersected with $U_i$ is a horizontal plane.  

The foliation $\mathcal F$ is $C^k$ if the charts $(U_i,\phi_i)$ can be chosen so that each $\phi_i$ is a $C^k$ diffeomorphism.  

The foliation $\mathcal F$ is $C^{k,0}$ if the charts $(U_i,\phi_i)$ can be chosen so that the restriction of each $\phi_i$ to a horizontal plane is a $C^k$ immersion and so that the tangent planes of the leaves vary continuously.
\end{definition}

Notice that $T\mathcal F$ exists and is continuous if and only if $\mathcal F$ is $C^{1,0}$.  Different amounts of transverse smoothness can also be specified by defining $C^{k,l}$ foliations, $k\ge l$.  For the purposes of this paper it is enough to know that $C^{k,1}$ foliations are necessarily $C^1$.  See Definition~2.1 of \cite{KR3} for the general definition.  

\begin{definition} \cite{KR2} \label{flowboxdefn} 
Let $\mathcal F$ be either a $C^k$ or $C^{k,0}$ foliation, and let $\Phi$ be a smooth transverse flow.  A {\sl flow box}, $F$, is an $(\mathcal F,\Phi)$ compatible closed chart, possibly with corners.  That is, it is a submanifold diffeomorphic to $D\times I$, where $D$ is either a closed $C^k$ disk or polygon (a closed disk with at least three corners), $\Phi$ intersects $F$ in the arcs $\{(x,y)\}\times I$, and each component of $D\times \partial I$ is embedded in a leaf of $\mathcal F$.  The components of $\mathcal F\cap F$ give a family of $C^k$ graphs over $D$.  
\end{definition}

A {\sl flow box decomposition} is a finite cover of $M$ by flow boxes so that their interiors are pairwise disjoint and intersections along boundaries satisfy certain conditions.  The constraints imposed on boundary intersections are not important for what follows, but can be found in Definition~3.1 of \cite{KR4}.  A flow box compatible isotopy is an isotopy which fixes setwise the flow boxes, and the cells of flow boxes, in a given flow box decomposition.  Given $(M,\mathcal F,\Phi)$, there is a flow box decomposition of $M$, and a flow box decomposition of a submanifold of $M$ always extends to a flow box decomposition of $M$ (Proposition~3.2 of \cite{KR4}).

There are two natural definitions of a {\sl transversal to $\mathcal F$.} 

\begin{definition} Let $\mathcal F$ be a $C^0$ foliation of a closed manifold.  A {\sl topological transversal} $\gamma$ is a curve which is {\sl topologically} transverse to $\mathcal F$; namely, no nondegenerate sub-arc of $\gamma$ isotopes relative to its endpoints into a leaf of $\mathcal F$.  
\end{definition}

\begin{definition} Let $\mathcal F$ be a $C^{1,0}$ foliation of a closed manifold.  A curve $c$ is {\sl smoothly transverse} to $\mathcal F$ if it is smooth and $T\gamma$ and $T\mathcal F$ span $TM$ at each point of $c$.  For brevity, such a curve is called a {\sl transversal} to $\mathcal F$.
\end{definition}

Notice that a smoothly embedded topological transversal is not necessarily a transversal.

\begin{lemma}\label{isotopetosmooth} 
Let $\mathcal F$ be a $C^1$ foliation of $M$, and let $\gamma$ be a topological transversal to $\mathcal F$ passing through a point $p\in M$.  There is an isotopy of $M$ relative to $p$ taking $\gamma$ to a transversal through topological transversals to $\mathcal F$.
\end{lemma}

\begin{proof} 
Pick a $C^1$ diffeomorphism taking a flow box to a horizontally foliated subset of $\mathbb R^3$.  Then, in coordinates, a straight line can connect any point of the lower boundary to any point of the upper boundary.  Since the straight line can be modified near the end points to be tangent to any non-horizontal vector, the segments can be glued to give a $C^1$ transverse curve in the manifold.  This curve can then be $C^1$ approximated with a smooth transverse curve.  
\end{proof}

Therefore, when $\mathcal F$ is $C^1$, there is a transversal through $p$ if and only if there is a topological transversal through $p$.  Moreover, when $\mathcal F$ is $C^1$, there is a transversal through every point if there is a transversal through every leaf.  However, even for $C^{\infty,0}$ foliations, this is not true (Proposition~\ref{no vp flow}), and hence Definition~2.11 of \cite{G1} gives rise to three distinct notions of tautness.

\begin{definition} A $C^0$ foliation $\mathcal F$ of a closed 3-manifold is {\sl topologically taut} if for every leaf $L$ of $\mathcal F$ there is a simple closed topological transversal to $\mathcal F$ that has nonempty intersection with $L$.
\end{definition}

\begin{definition} A $C^{1,0}$ foliation $\mathcal F$ is {\sl smoothly taut} if for every leaf $L$ of $\mathcal F$ there is a simple closed transversal to $\mathcal F$ that has nonempty intersection with $L$.
\end{definition}

\begin{definition} Let $\mathcal F$ be a $C^{1,0}$ foliation of a closed manifold.  The foliation $\mathcal F$ is {\sl everywhere taut}, or simply {\sl taut}, if for every point $p$ of $M$ there is a simple closed transversal to $\mathcal F$ that contains $p$.  
\end{definition}

 In the absence of sufficient smoothness, these three notions of tautness differ, and they are frequently confused in the literature.  Compare Theorem~2.1 of our paper \cite{KR2} with Theorem~\ref{vp flow} for one of many instances of this. 
In practice, everywhere taut is the most useful form of tautness.  Referring to everywhere taut foliations as taut allows most theorems to be stated without additional hypotheses.

In Section~\ref{Examples}, we give examples showing that the inclusions of foliations $$\{\mbox{everywhere taut}\}\subset \{\mbox{smoothly taut}\}\subset\{ \mbox{topologically taut}\}$$ are proper, even when restricting to $C^{\infty,0}$ foliations.

Restricting to $C^1$ foliations, these proper inclusions become equalities.

\begin{lemma}\label{C1} For $C^1$ foliations, topologically taut implies smoothly taut, and smoothly taut implies everywhere taut.
\end{lemma}

\begin{proof} Suppose $\mathcal F$ is topologically taut.  By Lemma~\ref{etautlemma}, there is a topological transversal through every point $p$ of $M$.  By Lemma~\ref{isotopetosmooth}, there is therefore a smoothly transverse transversal through every point.
\end{proof}

We include a fourth notion of tautness, introduced by Sullivan in \cite{Sullivan2}, since it motivates the usage of ``taut''.

\begin{definition} Let $\mathcal F$ be a $C^{k,0}$ foliation of a closed manifold with $k \ge 2$.  The foliation is {\sl geometrically taut} if there exists a metric on $M$ such that every leaf of $\mathcal F$ is a minimal surface.  An excellent reference for this material is the Appendix of \cite{Hass}.  
\end{definition}

\section{Standard foliation lemmas} \label{classicallemmas}

We begin by translating three classical results into the context of $C^0$ foliations.  In each case, the original proof translates immediately to yield the claimed result.  The proofs of each of the following three lemmas can be found, for example, in the proof of Proposition~4, Chapter VII, of \cite{CN}, and are similar.  For completeness, we illustrate the key ideas involved by including a proof of the first lemma.

\begin{lemma}\label{join} If $\mathcal F$ is a topologically taut $C^0$ foliation of $M$, then there is a connected closed topological transversal that intersects every leaf of $\mathcal F$.
\end{lemma}

\begin{proof} Let $\gamma$ be a collection of topological transversals that intersect every leaf of $\mathcal F$.  Suppose $\gamma$ is not connected.  Since $M$ is connected and the set of points on leaves that intersect a single closed transversal is an open set, there are distinct components of $\gamma$ that intersect a common leaf $L$ at points $p$ and $q$ respectively.  Let $\alpha$ be an arc in $L$ connecting $p$ and $q$, and let $D\subset L$ be a disk neighborhood of $\alpha$.  The foliation restricts to a product homeomorphic to $D \times I$ near $D$.  The two transversals can be replaced by a connected transversal by inserting a half twist in $D\times I$.  
\end{proof}

\begin{lemma} \label{top everywhere taut}
Suppose $\mathcal F$ is $C^0$ and topologically taut.  For every point $p$ in $M$, there is a topological transversal through $p$.\qed
\end{lemma} 

\begin{lemma} \label{etautlemma} If $L$ is a noncompact leaf of a $C^0$ foliation, then there is a topological transversal that has nonempty intersection with $L$.  \qed 
\end{lemma} 

A {\sl Reeb component} is a foliation of a solid torus whose boundary is a leaf and such that all other leaves are homeomorphic to planes.  See Example~\ref{phantomReeb}.  A foliation is {\sl Reebless} if it contains no Reeb component.  

For $\mathcal F$ a transversely oriented foliation, a {\sl dead end component} \cite{Thurstonthesis} $C$ of $\mathcal F$ is a connected submanifold of $M$ that is cobounded by a finite collection of torus leaves $T_1,...T_n$ of $\mathcal F$ so that, for one of the two choices of transverse orientation of $\mathcal F$, $C$ lies on the positive side of each $T_i$.  

The next result is well known and can be proved using ideas found in \cite{Goodman}.

\begin{prop} [Theorem~1, \cite{solodov}] \label{suegoodman}
A transversely oriented $C^0$ foliation $\mathcal F$ is topologically taut if and only if it contains no dead end components.
\end{prop}

 These results are not true if topological transversal is replaced by smooth transversal.  One issue that arises is the existence of {\sl phantom Reeb components} and, more generally, {\sl phantom dead end components.} 

\section{Phantom leaves}\label{Examples}

In this section, we give examples that highlight some differences between $C^1$ foliations and foliations with only continuous tangent plane field.  We begin by producing examples of Reebless $C^{1,0}$ foliations, and even Reebless $C^{\infty,0}$ foliations, that contain the tangent plane field of a compressible torus.

\begin{definition} 
Let $\mathcal F$ be a $C^{1,0}$ foliation.  A {\sl phantom} surface of $\mathcal F$ is a $C^1$ embedded surface in $M$ that is not contained in a leaf of $\mathcal F$, but has tangent plane field contained in $T\mathcal F$.
\end{definition}

\begin{figure}[htbp] % figure placement: here, top, bottom, or page
\centering
\includegraphics[width=3.75in]{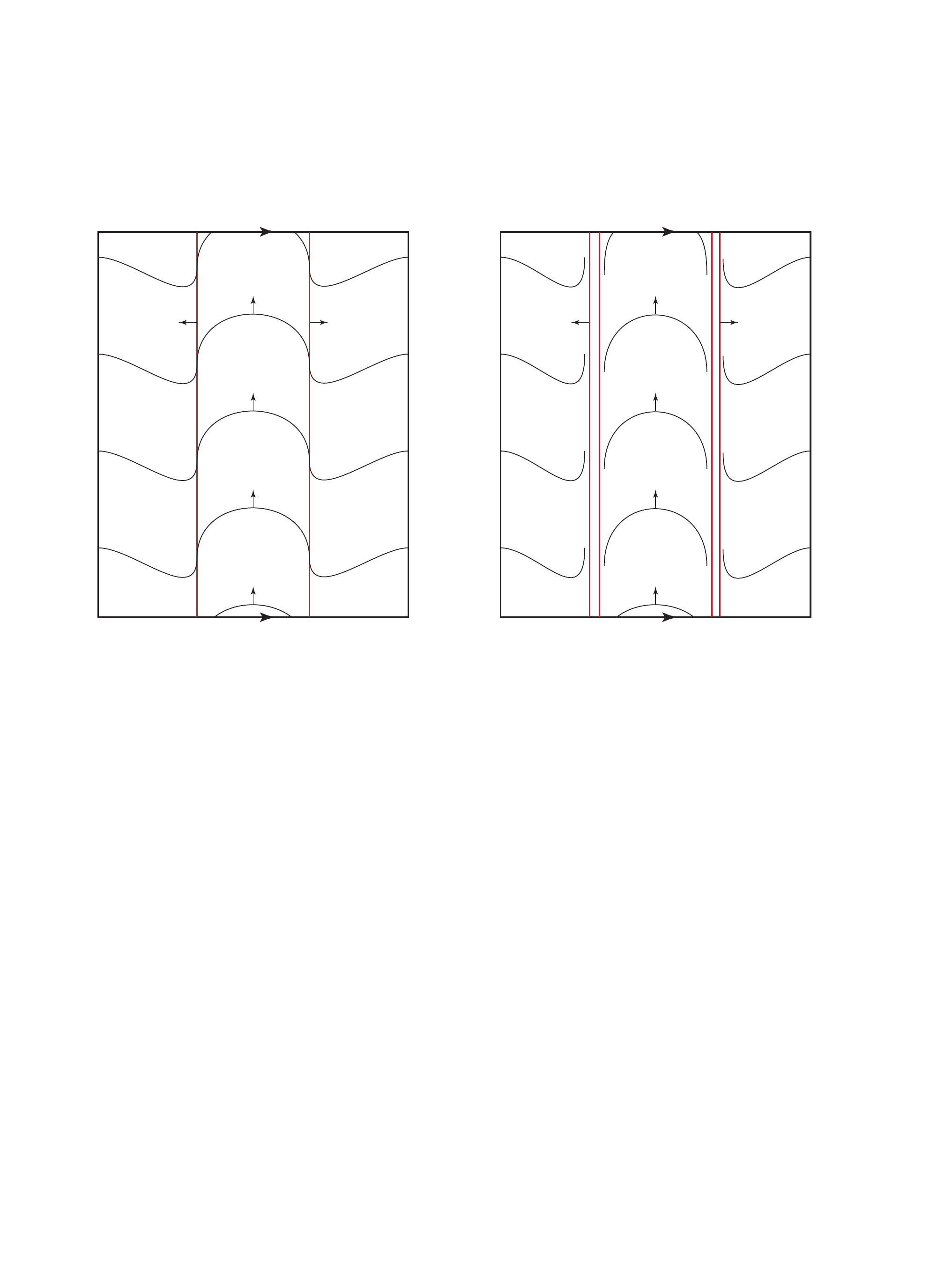} 
\caption{A Reeb-like annulus and a Reeb annulus with a collar on the compact leaves.}
\label{tangent torus}
\end{figure}

\begin{example}[Phantom Reeb component] \label{phantomReeb} Let $A = [-1,1] \times S^1$ be the annulus shown in Figure~\ref{tangent torus}, and let $\theta=0 \in S^1$.  Let $a_0$ be a smooth embedded arc in $A$ with boundary $\{\pm1\} \times \{0\}$ that is symmetric when the first coordinate is negated, is tangent to $ \{\pm 1/2\} \times S^1$ and is transverse to $\{x\} \times S^1$ for $x \neq \pm 1/2$.  As $\theta$ varies over $S^1$, let $a_\theta$ be the result of translating $a_0$ in the $S^1$ coordinate by $\theta$.  

The union of the $a_\theta$ is a $C^{\infty,0}$ transversely orientable foliation of $A$.  The curves $\{\pm 1/2\} \times S^1$ are integral curves for the foliation, and the transverse orientation may be chosen so that it points out of $S^1 \times[-1/2, 1/2]$.  Call the foliated annulus $[-1/2,1/2] \times S^1$ a {\sl phantom Reeb annulus}.  

Rotating $A$ about $\{0\} \times S^1$ produces a $C^{\infty, 0}$ Reeb-like foliation, called a {\sl phantom Reeb component}, $\mathcal R$ of $D^2 \times S^1$.  The torus, $T$, of radius $1/2$ about $\{0\} \times S^1$, is an example of a phantom torus.

Figure~\ref{tangent torus} also shows a foliation that when rotated about $\{0\} \times S^1$ is both $C^0$ close to $\mathcal R$ and has an actual Reeb component.  This foliation also has a product foliation by tori in a small neighborhood of the boundary of the Reeb component.
\end{example}

\begin{example} [$C^{\infty,0}$ and topologically taut, but not smoothly taut] \label{top not smooth}
Produce a foliation on $A = [-1,1] \times S^1$ first by letting $\{\pm 1\} \times S^1$ be leaves.  Let $\theta=0 \in S^1$, and let $b_0$ be a non-compact smooth arc that agrees with $a_0$ on $ [-1/2,1/2] \times S^1$, is transverse to $\{x\} \times S^1$ for $x \in (-1,-1/2) \cup (1/2,1)$, limits on $\{1\} \times S^1$ in the direction of increasing $\theta$, and limits on $\{-1\} \times S^1$ in the direction of decreasing $\theta$.  For other $\theta \in S^1$ let $b_\theta$ be given by translating $b_0$ through an angle of $\theta$ in the $S^1$ coordinate.  See Figure~\ref{topologically taut}.

Identifying each pair of points $( \pm 1, \theta)$ gives a $C^{\infty,0}$ transversely orientable foliation $\mathcal T$ of a torus.
\end{example}

\begin{figure}[htbp] % figure placement: here, top, bottom, or page
\centering
\includegraphics[width=1.75in]{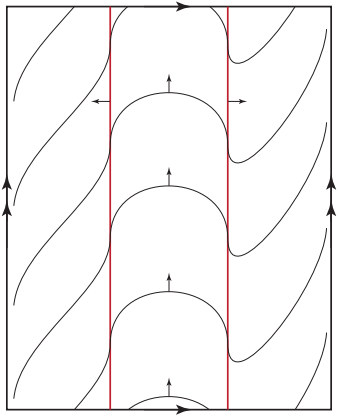} 
\caption{Topologically, but not smoothly, taut}
\label{topologically taut}
\end{figure}

\begin{prop} The foliation $\mathcal T$ is topologically taut, but not smooth\-ly taut.
\end{prop}

\begin{proof} It is straightforward to show that $\mathcal T$ is topologically taut.  Since the single compact leaf, $\{\pm 1\} \times S^1$ of $\mathcal T$ is parallel to the Reeb-like annulus, no smooth closed transversal can intersect it.
\end{proof}

Next we note that a transverse torus can often be used to create a phantom torus.  

\begin{prop} \label{phantomtori} 
Let $\mathcal F$ be a $C^{1,0}$ foliation of $M$.  Suppose there is a smoothly embedded torus $T$ in $M$ that is smoothly transverse to each leaf of $\mathcal F$, with $\mathcal F\cap T$ everywhere taut.
Then there is an isotopy of $M$, supported on a small neighborhood of $T$, that takes $\mathcal F$ to a $C^{1,0}$ foliation with phantom torus $T$.
\end{prop}

\begin{proof} 
Set $\lambda=\mathcal F\cap T$.  Since $\lambda$ is everywhere taut (as a foliation of $T$) it can be isotoped to be transverse to the smooth product foliation $\{t\}\times S^1$, for some choice of smooth product structure $ S^1\times S^1$ on $T$.  Hence, there is an isotopy of $M$, supported on a small neighborhood of $T$, and 
a smooth parametrization $(x,y,z)\in S^1\times [-1,1] \times S^1$ of a smaller neighborhood of $T$, so that $T$ is given by $y=0$, $\partial/\partial y$ is tangent to $\mathcal F$, $\mathcal F$ meets $S^1\times [-1,1] \times S^1$ in the product foliation $\lambda\times [-1,1]$, where $[-1,1]$ describes the range of the $y$-coordinate, and $\partial/\partial z$ is smoothly transverse to $\mathcal F$.

Let $b(y)$ be a damped version of $y^{1/3}$.  That is, it is $0$ away from $0$ and has a vertical tangency at $y=0$.  Now consider the homeomorphism $(x, y, z) \to (x, y, z + b(y))$.  It takes $\mathcal F$ to a new foliation tangent to $T$.
\end{proof}

Let $c$ be a transversal to a $C^{\infty,0}$ foliation $\mathcal F$.  Applying Proposition~\ref{phantomtori} to $\mathcal F$ and $T=\partial N(c)$, where $N(c)$ is a sufficiently small smooth regular neighborhood of $c$, yields examples of $C^{\infty,0}$ foliations of $M$ that are smoothly taut, but not everywhere taut.  

\begin{example} [$C^{\infty,0}$ and smoothly taut, but not everywhere taut] \label{phantomReebinsert} Let $\mathcal F$ be a smoothly taut, transversely oriented, $C^{\infty,0}$ foliation of $M$ and let $c$ be a smooth closed transverse curve.  Identify $N(c)$ and $\mathcal F|N(c)$ with $D^2 \times S^1$ foliated by disks.  Produce a $C^{\infty, 0}$ foliation $\mathcal H$ by replacing $\mathcal F$ on $D^2 \times S^1$ by $\mathcal R$.
\end{example}

\begin{prop}\label{no vp flow} The foliations of Example~\ref{phantomReebinsert} are smoothly taut, but they are not everywhere taut.  They admit no closed dominating 2-form, and they are not transverse to any volume preserving flow for any choice of Riemmanian metric.
\end{prop}

\begin{proof} Since $T$ is a separating integral surface, no smooth closed trans\-versal can intersect a point of $T$.  The existence of a dominating 2-form would contradict Stokes theorem applied to $T$.  Since $T$ separates, a volume preserving flow transverse to $\mathcal H$ can not exist.
\end{proof} 

The construction of Example~\ref{phantomReebinsert} will be used in Corollary~\ref{limits of Reeb} to show that all foliations are, up to isotopy, limits of foliations with Reeb components.

Each of the preceding examples took advantage of a phantom Reeb component.  More generally, one can construct examples with one or more phantom dead end components.  

\begin{definition}
For $\mathcal F$ a transversely oriented foliation, a {\sl phantom dead end component} $C$ of $\mathcal F$ is a connected submanifold of $M$ that is cobounded by a finite collection of tori $T_1,...T_n$, where $T_1$ is a phantom leaf and each $T_j,j>2,$ is either a leaf or a phantom leaf for $\mathcal F$, so that, for one of the two choices of transverse orientation of $\mathcal F$, $C$ lies on the positive side of each $T_i$.  
\end{definition}

Notice that if $C$ is a phantom dead end component of $\mathcal F$, cobounded by tori $T_1,...,T_n$, and $T_j$ is a leaf of $\mathcal F$, then there is no transversal through $T_j$.

\begin{example} [Phantom dead ends obstruct smooth tautness] \label{examplelala} 
Let $\mathcal F$ be a topologically taut $C^{\infty,0}$ foliation containing a torus leaf $T_1$.  Suppose further that there is a torus $T_2$ isotopic to $T_1$ that meets $\mathcal F$ transversely in a $C^1$ foliation without Reeb annuli.  Proposition~\ref{phantomtori} can be applied to $\mathcal F$ and $T_2$ to obtain a $C^{\infty,0}$ foliation $\mathcal G$ so that $T_1$ is a leaf of $\mathcal G$, and $T_1$ and $T_2$ cobound a dead end component $T\times I$ of $\mathcal G$.  Since $\mathcal F$ is topologically taut, so is $\mathcal G$.  However, since there is no transversal through $T_1$, the foliation $\mathcal G$ is not smoothly taut.

There are many examples of foliated manifolds $(M,\mathcal F)$ satisfying these conditions.  For example, begin with a manifold $X$ with torus boundary and a topologically taut foliation $\mathcal F_0$ transverse to $\partial X$.  Let $\mathcal F_1$ be the foliation of $X$ obtained from $\mathcal F_0$ by adding as leaf $T_1=\partial X$, and modifying the leaves of $\mathcal F_0$ in a collar of $\partial X$ so that they spiral about $\partial X$.  Let $(M,\mathcal F)$ be the double of $(X,\mathcal F_1)$.
\end{example}

In fact, even with the hypotheses of smooth tautness, the usual argument (see Lemma~\ref{join}) for combining several smooth transversals produces a single curve that may be only topologically transverse.  This is illustrated by the following family of examples.

\begin{example} [$C^{\infty,0}$ and smoothly taut, but there is no transversal that has nonempty intersection with every leaf] Begin with a 3-manifold $X$ with torus boundary and smoothly taut foliation $\mathcal F_0$ transverse to $\partial X$.  Choose $(X,\mathcal F_0)$ so that $\mathcal F_0$ has minimal set disjoint from $\partial X$.  Now let $(M,\mathcal F)$ be the double of $(X,\mathcal F_1)$.  There are disjoint tori $T_1$ and $T_2$ in $M$ parallel to $\partial X$, and these can be chosen to lie transverse to $\mathcal F$.  Applying Proposition~\ref{phantomtori} to $\mathcal F$ and the tori $T_1$ and $T_2$ yields a $C^{\infty,0}$ foliation $\mathcal G$ that is smoothly taut and has phantom dead end component.  However, there is no connected transversal that has nonempty intersection with each leaf of $\mathcal G$.
\end{example}

\begin{example} The previous examples are built around the separating properties of certain phantom tori.  The same effects can be achieved without tangent tori.  Figure~\ref{crease} shows an alternate version of the first foliation described in Example~\ref{phantomReeb}.
\end{example}

\begin{figure}[htbp] % figure placement: here, top, bottom, or page
\centering
\includegraphics[width=1.75in]{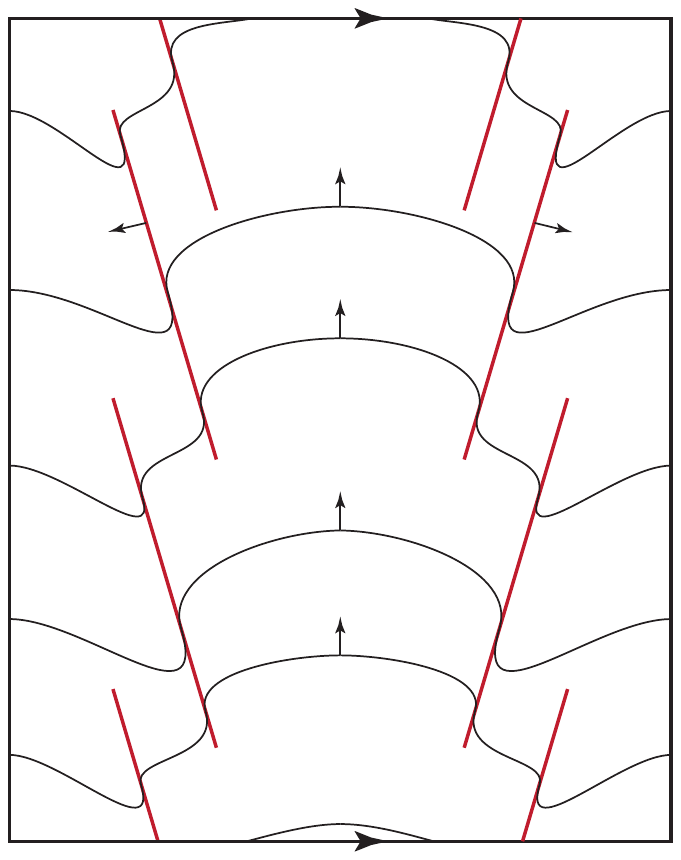} 
\caption{}
\label{crease}
\end{figure}

\section{Approximating foliations by foliations}
Next we show that topologically taut foliations can be $C^0$ approximated by isotopic smoothly taut foliations which can in turn be approximated by isotopic everywhere taut foliations.  Thus in seeking geometric applications of foliations, it is enough to prove the existence of a topologically taut foliation.

\begin{definition} A {\sl plaque} of a flow box $F$ is a connected component of a leaf of $\mathcal F$ intersected with $F$.  A {\sl plaque neighborhood in $F$} is a connected closed set with nonempty interior that is a union of plaques.  
\end{definition}

The next lemma is immediate by Proposition~3.10 of \cite{KR4}.

\begin{lemma}\label{create U} Suppose $\mathcal F$ is a $C^{\infty,0}$ foliation, and $U$ is a small open regular neighborhood of a finite collection of pairwise disjoint plaque neighborhoods.  There is a $C^0$ small isotopy of $M$ taking $\mathcal F$ to a $C^0$ close $C^{\infty,0}$ foliation that is smooth when restricted to $U$.  \qed
\end{lemma}

\begin{lemma}\label{smoothly connect} Let $F=D^2 \times I$ be a flow box for $(\mathcal F, \Phi)$.  Suppose $\mathcal F$ is smooth on a plaque neighborhood $U$ of $F$.  If $p \in D^2 \times \{0\}$ and $q \in D^2 \times \{1\}$, then there exists a transversely smooth arc $\alpha$ agreeing with $\Phi$ in a neighborhood of $p$ and $q$.  
\end{lemma}

\begin{proof} Choose a segment of $\Phi$ that starts at $p$ and ends at point $u_1 \in U$.  Pick another segment that starts at $q$ and ends at $u_2 \in U$ where $u_2$ lies in a plaque just above the plaque containing $u_1$.  These segments can be combined with a transverse arc connecting $u_1$ and $u_2$ to produce the desired smooth arc $\alpha$.  
\end{proof}

\begin{thm}\label{top approx by smooth} Given a topologically taut $C^{1,0}$ foliation $\mathcal F$, there is a $C^0$ small isotopy of $M$ taking $\mathcal F$ to a $C^0$ close, smoothly taut, $C^{\infty, 0}$ foliation.
\end{thm}

\begin{proof} Let $\gamma$ be a collection of topological transversals that intersect every leaf of $\mathcal F$.  Choose a collection of flow boxes that cover $\gamma$, that intersect only along their horizontal boundaries, and that have vertical boundaries disjoint from $\gamma$.  Extend these to a flow box decomposition $\mathcal B$ of $M$.

Choose a plaque neighborhood in each flow box of $\mathcal B$, and let $U$ be a small regular neighborhood of their union.  Apply Lemma~\ref{create U} to obtain an isotopy taking $\mathcal F$ to a foliation smooth on $U$, and denote the image of $\gamma$ under this isotopy by $\gamma'$.  Lemma~\ref{smoothly connect} can be used to replace $\gamma'$ with a transversely smooth curve $c$ one flow box at a time.
\end{proof}

\begin{lemma}\label{everywhere} Let $\mathcal B$ be a flow box decomposition for an topologically taut $C^{1,0}$ foliation $\mathcal F$.  There is a $\mathcal B$ compatible isotopy taking $\mathcal F$ to a $C^0$ close $C^{\infty,0}$ foliation $\mathcal G$ for which there exists a connected closed transversely smooth curve $c$ such that $c$ intersects every plaque of every flow box.
\end{lemma}

\begin{figure}[h] 
\centering
\includegraphics[width=3in]{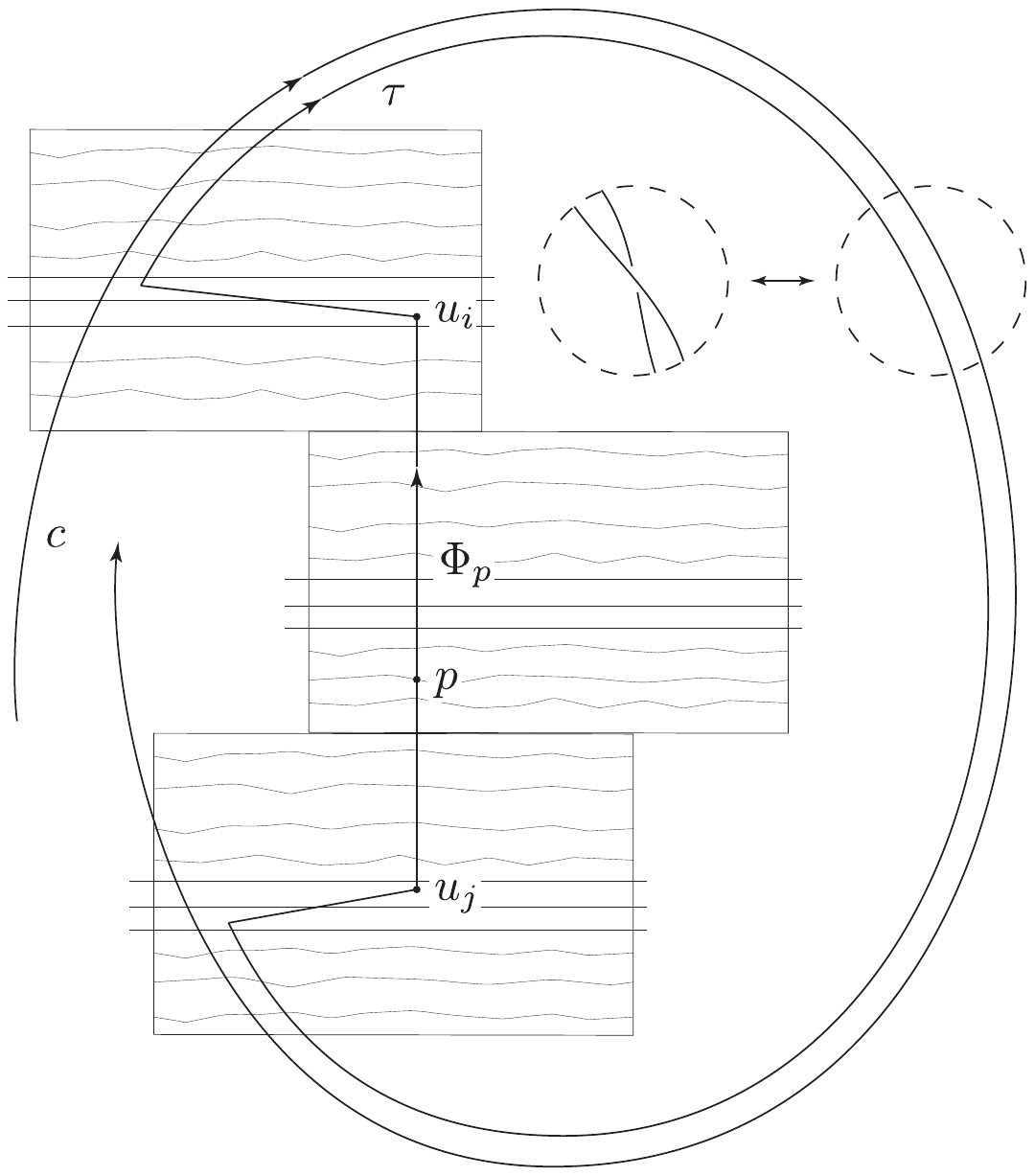} 
\caption{}
\label{transversals}
\end{figure}

\begin{proof} By Theorem~4.1 and Proposition~3.10 of \cite{KR4} there is a $\mathcal B$ compatible isotopy taking $\mathcal F$ to $\mathcal G$, where $\mathcal G$ is $C^{\infty,0}$ and smooth on a set $U$ that contains a plaque neighborhood of every flow box of $\mathcal B$.  By Lemma~\ref{top everywhere taut}, there is a closed topological transversal $\gamma_i$ intersecting $U_i$, for each component of $U_i$ of $U$.  Using the method of Lemma~\ref{join}, construct a connected topological transversal $\gamma$ from the $\gamma_i$ so that $\gamma$ intersects every component of $U$.

Now replace $\gamma$ with a smooth closed transversal $c$ by performing a small isotopy that preserves $\mathcal B$ and smoothness of the foliation on $U$.  To do this, first choose small flow boxes containing $\gamma$ and subordinate to $\mathcal B$, and then apply the smoothing operation of Lemma~\ref{top approx by smooth}.

Let $F$ be a flow box and choose $p\in F$.  As shown in Figure~\ref{transversals}, consider the segment $\Phi_p$ of $\Phi$ that contains $p$, flows up and down, crosses both horizontal boundary components of $F$, and ends flowing up at $u_i \in U_i$ and flowing down at $u_j \in U_j$.  Let $\tau$ be a small push off of a segment of $c$ that starts in $U_i$ flowing up and ends in $U_j$.  Smoothly joining $\tau$ and $\Phi_p$ through $U_i \cup U_j$ produces a smooth transversal $c_p$ that passes through every plaque of $F$.  

The original transversal $c$ can be smoothly connected with $c_p$ using a half twist.  Repeat this construction for every flow box of $\mathcal B$.
\end{proof}

\begin{cor}\label{top approx} Given a topologically taut $C^{1,0}$ foliation $\mathcal F$, there is a $C^0$ small isotopy of $M$ taking $\mathcal F$ to a $C^0$ close, everywhere taut, $C^{\infty, 0}$ foliation.
\end{cor}

\begin{proof} Since there was no restriction on the choice of $p$ in the proof of Lemma~\ref{everywhere}, this follows immediately.
\end{proof}

The next corollary explores the limiting interaction between foliations with and without Reeb components.

\begin{cor}\label{limits of Reeb} A $C^{0}$ Reebless foliation is isotopic to a $C^0$ limit of $C^{\infty,0}$ foliations with Reeb components.  \qed
\end{cor}

\begin{proof}
Let $\mathcal F$ be a $C^0$ Reebless foliation, and let $\gamma$ be a closed topological transversal to $\mathcal F$.  To see that such a curve exists, pick a flow transverse to $\mathcal F$, and consider an arc of the flow that starts and ends in a single flow box.  The method of Lemma~\ref{join} can be used to replace the arc with a topological transversal.

Theorem~\ref{top approx by smooth} is stated for taut foliations, but the method of proof allows a single topological transversal, $\gamma$, to be isotoped to a transversal $c$.  Thus applying Theorem~\ref{top approx by smooth} produces an isotopic $C^{\infty,0}$ foliation and a transversal $c$.

The method of Example~\ref{phantomReebinsert} produces an isotopic $C^{\infty,0}$ foliation $\mathcal H$ that is the $C^0$ limit of foliations with Reeb components.
\end{proof}

\section{Dominating 2-forms and volume preserving flows}

For completeness we include a $C^{1,0}$ version of Sullivan's theorem \cite{Sullivan2}.  

\begin{thm}\label{vp flow} Let $\mathcal G$ be an everywhere taut, transversely oriented, $C^{1,0}$ foliation.  There exists a smooth, closed 2-form $\omega$ on $M$ such that $\omega$ is positive on $T\mathcal G$.  Moreover, fixing any Riemannian metric on $M$, there is a smooth volume preserving flow $\Phi$ transverse to $\mathcal G$.
\end{thm}

\begin{proof} Let $\beta$ be a smooth 2-form on a disk $D$ such that $\beta$ is $0$ near $\partial D$ and is otherwise positive on $TD$.  Let $\pi:S^1 \times D^2 \to D^2$ be projection so that $\pi^*\beta$ is a closed form on $S^1 \times D^2$.

Given a transversal $\gamma$ to $\mathcal G$, let $N(\gamma)$ be a small solid torus neighborhood foliated by disks of $\mathcal G|N(\gamma)$.  Choose a diffeomorphism $h:N(\gamma) \to S^1 \times D^2$ that maps leaves of $\mathcal G|N(\gamma)$ to disks that are transverse to the first coordinate.  Then $h^*\pi^*\beta$ is positive on $T\mathcal G$ in a neighborhood of $\gamma$ and non-negative at all points of $N(\gamma)$.

Since $M$ is compact and $\mathcal G$ is everywhere taut, there is a finite collection of smooth transversals $\gamma_i$ such that at every point of $M$, at least one of $h_i^*\pi^*\beta$ is positive on $T\mathcal G$.  If follows that $\omega = \Sigma_i h_i^*\pi^*\beta$ has the desired properties.

Let $\Omega$ be a volume form on $M$.  The equation $\omega = X \lrcorner \Omega$ uniquely determines a vector field $X$ that is transverse to $\mathcal G$.  Let $\Phi$ be the associated flow for $X$.  By Cartan's formula,
$$\mathcal{L}_X \Omega = X \lrcorner d\Omega + d(X\lrcorner\Omega) = d(X\lrcorner \Omega) = d\omega =0,$$ and it follows that $\Phi$ preserves volume.
\end{proof}

\begin{cor}\label{topdomform}
Let $\mathcal F$ be a topologically taut, transversely oriented, $C^{1,0}$ foliation.  There is a $C^0$ small isotopy of $M$ taking $\mathcal F$ to a $C^0$ close, everywhere taut, $C^{\infty, 0}$ foliation $\mathcal G$ satisfying
\begin{enumerate}
\item there exists a smooth, closed 2-form $\omega$ on $M$ such $\omega$ is positive on $T\mathcal G$, and, 
\item fixing any Riemannian metric on $M$, there is a smooth volume preserving flow $\Phi$ transverse to $\mathcal G$.
\end{enumerate}
\end{cor}

\begin{proof} 
This follows immediately from Corollary~\ref{top approx} and Theorem~\ref{vp flow}.
\end{proof}

\section{Approximating foliations by contact structures}

In this section we contrast the main result of \cite{bowden} and \cite{KR3}, which gives properties of any  contact   approximation of an everywhere taut foliation with the corresponding result, Theorem~\ref{limit of both}, for topologically taut foliations.

\begin{thm} [Theorem~1.2 of \cite{bowden} and Theorem~1.2, \cite{KR3}, approximation without tautness] \label{approx} Let $M$ be a closed, connected, oriented 3-manifold, and let $\mathcal F$ be a transversely oriented $C^{1,0}$ foliation on $M$.  Then $\mathcal F$ can be $C^0$ approximated by a positive (respectively, negative) contact structure if and only if $\mathcal F$ is not a foliation of $S^1\times S^2$ by spheres.
\end{thm}

\begin{thm}[\cite{bowden,KR2, KR3}]\label{everwheretautimplies} Let $\mathcal G$ be an everywhere taut, $C^{1,0}$ foliation on a manifold other than $S^2 \times S^1$ and let $\Phi$ be a transverse volume preserving flow.  Then $\mathcal G$ can be $C^0$ approximated by a positive (respectively, negative) contact structure.  Moreover, any contact structure that is transverse to $\Phi$ is weakly symplectically fillable and universally tight.
\end{thm}

\begin{cor} \label{topologicallytautimplies}
Let $\mathcal F$ be a topologically taut, $C^{1,0}$ foliation on a manifold other than $S^2 \times S^1$.  Then $\mathcal F$ can be approximated by a a pair of contact structures $\xi_{\pm}$, $\xi_+$ positive and $\xi_-$ negative, such that $(M,\xi_+)$ and $(-M,\xi_-)$ are weakly symplectically fillable and universally tight.
\end{cor}

\begin{proof}
This follows immediately from Corollary~\ref{topdomform} and Theorem~\ref{everwheretautimplies}.
\end{proof}

 When the condition on a foliation is weakened from everywhere taut to smoothly taut, 
 it no longer follows that any positive contact structure sufficiently close to $\mathcal F$ must be weakly symplectically fillable and universally tight.

\begin{thm}\label{limit of both} 
There exist $C^{\infty,0}$ smoothly taut, transversely oriented foliations which can be $C^0$ approximated both by weakly symplectically fillable, universally tight contact structures and by overtwisted contact structures.
\end{thm}

\begin{proof} Let $\mathcal F$ be any transversely oriented, everywhere taut $C^{\infty,0}$ foliation, and let $c$ be a transversal to $\mathcal F$ that has nonempty intersection with every leaf of $\mathcal F$.  As described in Example~\ref{phantomReebinsert}, replace a regular neighborhood of $c$ foliated by disks by a phantom Reeb component to create a new $C^{\infty,0}$ foliation $\mathcal H$.  The foliation $\mathcal H$ is smoothly taut, but not everywhere taut.

By Corollary~\ref{topologicallytautimplies}, $\mathcal H$ can be approximated by a weakly symplectically fillable, universally tight contact structure.

The construction of $\mathcal H$ involved the creation of a phantom Reeb component near a smooth closed transversal.  Replacing a neighborhood of this phantom Reeb component with the $C^0$ close Reeb foliation described in Example~\ref{phantomReeb} gives a $C^0$ close $C^{\infty,0}$ foliation $\mathcal H_1$.  Thus $M$ can be written as the union of three codimension 0 pieces: $S$, a solid torus with the Reeb foliation, $T \times I$, the product of a torus and interval foliated as smooth product, and $C$, the closure of the complement of $(T \times I) \cup S$, foliated by the restriction of $\mathcal H_1$.

Set $I=[a,b]$, and consider the contact structures $\xi_{\epsilon,n}= \mbox{ker } \alpha_{\epsilon,n}$ on $T\times I$ given by 
$$\alpha_{\epsilon,n} = dz+\epsilon (cos \,nz \,dx + sin\, nz \,dy).$$
Such contact structures have constant slope characteristic foliations on each $T \times \{z\}$.  By constraining $\epsilon>0, n, a,$ and $b$ appropriately, we obtain a contact structure $\xi=\xi_{\epsilon,n}$ on $T \times I$ such that

\begin{enumerate}

\item $\xi$ has Giroux torsion greater than 1, 

\item $\xi$ strictly dominates $\mathcal F$ along $\partial(T \times I)$, and

\item $\xi$ is $C^0$ close to the product foliation by vertical tori.

\end{enumerate}

The first condition means that all slopes occur, at least once, as $z$ moves across $I$.  The second condition is that the slopes at $a$ and $b$ must be close to, and strictly dominate, in the sense of Definition~5.3 of \cite{KR2}, i.e., the slope of $\mathcal F \cap \partial(T \times I)$ when viewed from outside of $C \cup S$.  The third condition follows by choosing $\epsilon>0$ sufficiently small (see Proposition 2.3.1 \cite{ET}).

The next step is to extend $\xi$ over all of $S \cup C$.  Let $J=[a - \delta, b + \delta]$ and regard $\xi$ as defined on $T \times J$, a neighborhood of $T \times I$ in $M$ for which $\mathcal H_1 \cap \partial (T \times J) = \mathcal F \cap \partial (T \times J)$.  If $\delta$ is small enough, $\xi$ will dominate $\mathcal F \cap \partial(T \times J)$, or equivalently, $\xi$ dominates $\mathcal H_1 \cap \partial (T \times J)$.

Since $c$ has nonempty intersection with every leaf of $\mathcal F$, condition (2) is what is required to apply the techniques of \cite{KR2} to extend $\xi$ from $T \times J$ to the remaining portion $C\cup S$ while continuing to approximate $\mathcal H_1$.

Choose $z_0\in I$ so that the slope of the characteristic foliation of $\xi$ in $T \times\{z_0\}$ matches the slope of a compressing disk.  Then any closed curve of the characteristic foliation on $T\times\{z_0\}$ bounds an overtwisted disk.  
\end{proof}

Since Theorem~{3.4} of \cite{calegari} shows that any $C^0$ foliation can be isotoped to a $C^{\infty,0}$ foliation, Corollary~\ref{topologicallytautimplies} implies the following.

\begin{cor} [Corollary~1.7, \cite{KR3}] \label{L space}
An $L$-space does not admit a transversely orientable, topologically taut, $C^0$ foliation.  \qed
\end{cor}

\end{document}